\documentclass[11pt]{stijl}

\usepackage{proof}
\usepackage{amssymb} 
\usepackage{amsmath} 
\usepackage{stmaryrd}
\usepackage{euscript} 
\usepackage{latexsym}
\usepackage{mathrsfs} 
\usepackage{float} 
\usepackage{cancel}
\usepackage{color}
\usepackage{wasysym}
\usepackage[round]{natbib} 
\usepackage{chngcntr}

\counterwithin{figure}{section}
\newcommand{\Kf}{{\sf K4}}

\newcommand{\Sf}{{\sf S4}}
\newcommand{\GL}{{\sf GL}}

\newcommand{\Grz}{{\sf Grz}}

\newcommand{\G}{{\sf G}}
\newcommand{\GR}{{\sf GR}}

\newcommand{\Gp}{{\sf G3p}}
\newcommand{\GpR}{{\sf G3R}}

\newcommand{\GKf}{{\sf G3K4}}
\newcommand{\GGL}{{\sf G3GL}}

\newcommand{\GSf}{{\sf G3S4}}
\newcommand{\iGGL}{{\sf iG3GL}}
\newcommand{\iGKf}{{\sf iG3K4}}

\newcommand{\lgc}{{\sf L}}

\newcommand{\lang}{\ensuremath {{\EuScript L}}}

\newcommand{\defn}{\equiv _{\mbox{\em \tiny df}}} 
\newcommand{\af}{\vdash}
\newcommand{\afinf}{\vdash^\circ}

\newcommand{\adm}{\makebox{\raisebox{.4ex}{\scriptsize $\ \mid$}\raisebox{.28ex}{\footnotesize $\! \sim \,$}}}

\newcommand{\imp}{\rightarrow}

\newcommand{\en}{\wedge} 
\newcommand{\of}{\vee}

\newcommand{\bx}{\raisebox{-.01mm}{$\Box$}}
\newcommand{\dbx}{\raisebox{.13mm}{$\boxdot$}}

\newcommand{\bof}{\bigvee}
\newcommand{\ben}{\bigwedge}
\newcommand{\seq}{\Rightarrow}
\newcommand{\sml}{\preccurlyeq}

\newcommand{\smll}{\prec}

\newcommand{\De}{\Delta}
\newcommand{\Ga}{\Gamma}

\newcommand{\Sig}{\Sigma}
\newcommand{\sig}{\sigma}
\renewcommand{\phi}{\varphi}
\newcommand{\varchi}{\raisebox{1pt}{$\chi$}} 

\newcommand{\cald}{{\cal D}}

\newcommand{\calr}{{\cal R}}
\newcommand{\cals}{{\cal S}}

\newcommand{\itm}{\item[$\circ$]}

\newtheorem{Theor}{\bf Theorem}
\newenvironment{theorem}{\begin{Theor}\em }{\end{Theor}}
\newtheorem{Lemma}{\bf Lemma}
\newenvironment{lemma}{\begin{Lemma}\em }{\end{Lemma}}
\newtheorem{Coro}{\bf Corollary}
\newenvironment{corollary}{\begin{Coro}\em }{\end{Coro}}
\newtheorem{Remark}{\bf Remark}
\newenvironment{remark}{\begin{Remark}\em }{\end{Remark}}
\newtheorem{defin}{\bf Definition}

\newtheorem{exam}{\bf Example}

\newtheorem{notat}{\bf Notation}
 
\newenvironment{proof}{{\bf Proof}}{\hfill $\slot$}

\newcommand{\slot}{\hfill \large \sun}

\renewcommand{\textheight}{21cm}

\begin{document}
\title{Reasoning in circles}
\author{Rosalie Iemhoff \\
\normalsize Department of Philosophy  \\
\normalsize Utrecht University, the Netherlands \\
\normalsize r.iemhoff@uu.nl} 
\maketitle

\begin{center}
{\it For Albert, on the occassion of his 65th birthday}
\end{center}

\begin{abstract}
\noindent
  Circular proofs, introduced by Daniyar Shamkanov, are proofs in which 
  assumptions are allowed that are not axioms but do appear at least twice 
  along a branch. Shamkanov has shown that a formula belongs to the 
  provability logic $\GL$ exactly if it has a circular proof in the modal 
  logic $\Kf$. Shamkanov uses Tait style proof systems and infinitary proofs.  
  In this paper we prove the same result but then for sequent calculi and 
  without the detour via infinitary systems. We also obtain a mild generalisation 
  of the result, implying that its intuitionistic analogue holds as well. 
\end{abstract}

\begin{center}
\small{ \textit{Keywords:}  provability logic, sequent calculus, circular proofs} \\
\small{ \textit{MSC:}  03B45, 03F45, 03F07}
\end{center}

\section{Introduction}

In the spring of 2015 Lev Beklemishev told Albert Visser and me about a theorem that his then student Daniyar Shamkanov had proved, a theorem stating that when a certain form of circularity is allowed in proofs in transitive modal logic $\Kf$, the resulting proof system is sound and complete with respect to the provability logic $\GL$. Later that week Albert told me: ``Such a theorem makes me happy for the whole day''. For this Liber Amicorum, in honor of his 65th birthday and retirement, I take a closer look at this source of happiness. 

In \citep{shamkanov2014} develops a notion of {\em circular proof}, that extends  the standard notion of proof in a given Gentzen or Tait calculus by allowing derivations in which the leafs are not axioms but equal to a sequent below that leaf. For references to earlier occurrences of notions of circularity in proofs, see \citep{brotherston2006}. Clearly, proofs are special instances of circular proofs, namely those in which all leafs are axioms. Shamkanov shows that being provable in $\GL$ is equal to having a circular proof in $\Kf$. For example, the following is a circular proof of L\"ob's principle in the well-known sequent calculus for $\Kf$ as given in Section~\ref{seccalculi}.
\[
 \infer[R_\Kf]{\bx(\bx\phi\imp\phi) \seq \bx\phi}{
  \infer[L\!\imp]{\dbx (\bx\phi\imp\phi) \seq \phi}{
   \bx(\bx\phi\imp\phi) \seq \bx\phi & \bx(\bx\phi\imp\phi),\phi \seq \phi } }
\]
When I read Shamkanov's clever paper after Lev's visit to Utrecht I wondered whether his approach, which uses Tait style calculi, could be adapted for (two--sided) sequent calculi and whether it could be generalized to other logics. The answer to the first question is {\em yes}, and to the second question {\em I am not so sure}. In this paper I will explain why. 

The key ideas in the paper are certainly Shamkanov's, but I do present certain facts in a different way. Most importantly, I do not use infinitary proof systems as an intermediary step between standard and circular proof systems, as is done in \citep{shamkanov2014}.  

In trying to establish whether Shamkanov's Theorem could be generalized to other logics, I have tried, in this paper, to generalize the assumptions under which the main theorem holds.  And in doing so, I discovered that the results actually seem to very much depend on particular properties of the logic $\GL$. The only immediate corollary from the generalization is the insight that an  analogue of the main theorem holds for the intuitionistic versions of the modal logics (Theorem~\ref{thmmainint}).

\subsection{Structure of the paper}

We proceed as follows. If $\af_\G S$ denotes that $S$ has a proof in Gentzen calculus $\G$ and $\afinf_{\G}S$ denotes that $S$ has a circular proof in $\G$, then my translation of Shamkanov's Theorem in terms of the standard sequent calculi $\GKf$ and $\GGL$ for $\Kf$ and $\GL$, can be expressed as
\[
 \af_{\GGL}S \text{ if and only if }\afinf_{\GKf}S. 
\]
In this paper I generalize this to Theorem~\ref{thmmain}:
For every extension $\G$ of $\Gp$ by ordered box rules that are closed under weakening and contraction and for any slim box rules $R_1$ and $R_2$ such that $\GR_2$ is the circular companion of $\GR_1$:
\[
 \af_{\GR_1} S \text{ if and only if } \afinf_{\GR_2} S.
\]
The technical terms will be explained in the next sections, but let me mention here that all requirements are met by many sequent calculi for modal logics, except for the last requirement about modal companion. That $\GKf$ is the circular companion of $\GGL$ seems, as we will see, to depend strongly on the properties of $\GL$.  

Theorem~\ref{thmmain} consists of two directions, the one from left to right is Lemma~\ref{lemmain} and the other direction is Lemma~\ref{lemmaintoo}. Section~\ref{secresults} contains the main result and the application to intuitionistic modal logics. It also shows  why the obvious generalization to Grzegorczyk logic does not work. 

I thank two anonymous referees for useful comments on an earlier version of this paper.

\section{Logics and sequent calculi}
 \label{seccalculi}
The logics we consider are modal propositional logics, formulated in a language $\lang$ that contains constants $\top$ and $\bot$, propositional variables or atoms $p,q,r,\dots$ and the connectives $\en,\of,\neg,\imp$ and the modal operator $\bx$. The expression $\dbx\phi$ stands for $\phi\en\bx\phi$. All logics that we consider are extensions of classical propositional logics, but we do not assume them to be normal. 

We will mainly work with sequents, which are expression $\Ga \seq \De$, where $\Ga$ and $\De$ are finite multisets of formulas in $\lang$, that are interpreted as $I(\Ga \seq \De) = (\ben\Ga \imp \bof\De)$. We denote finite multisets by $\Ga,\Pi,\De,\Sig$. We also define ($a$ for antecedent, $s$ for succedent): 
\[
 (\Ga \seq \De)^a \defn \Ga \ \ \ \ (\Ga \seq \De)^s \defn \De.
\]
When sequents are used in the setting of formulas, we often write $S$ for $I(S)$, such as in $\af S$, which thus means $\af I(S)$. Multiplication of sequents is defined as   
\[
 S_1 \cdot S_2 \defn (S_1^a \cup S_2^a \seq S_1^s \cup S_2^s). 
\]
Given a multiset $\Ga$, we write $\bx\Ga$ for the multiset obtained by putting a box in front of every formula in $\Ga$, and $\dbx\Ga$ for $\Ga \cup \bx\Ga$. For a sequent $S$ we write $\bx S$ for the sequent $\bx S^a \seq \bx S^s$, and similarly for $\dbx S$. For example, $\bx(p\seq \,)$ denotes $(\bx p\seq \,)$. 

We will be interested in multisets in which the repetition of formulas occurs for certain formulas only. 
Given a multiset $\Ga$ we denote by $\Ga_{\dbx}$ the largest set $\dbx\Pi$ such that $\dbx\Pi \subseteq \Ga$, and the multiset $\Ga_{\dbx} \cup \{\phi \in \Ga \mid \phi\not\in \Ga_{\dbx} \}$ by $\Ga^*$. Thus $\Ga^*\backslash \Ga_{\dbx}$ is a set. 
With every sequent $S = (\Ga \seq \De)$ the {\em set--sequent} 
$S^* = (\Ga^*\seq \De^*)$ is associated. A sequent $S$ is a {\em set--sequent} if it is of the form $S^*_0$ for some sequent $S_0$. 
Two sequents are {\em set--equivalent} if their set--sequents are equal. 

Given formulas $\varchi(p)$ and $\phi$, $\varchi(\phi)$ denotes the result of replacing $p$ everywhere by $\phi$ in $\varchi$. For a multiset $\Ga$, we use $\varchi(\Ga)$ as 
abbreviation for the set of formulas $\{ \varchi(\phi) \mid \phi \in \Ga\}$ and 
$\varchi(S)$ for $\varchi(I(S))$. For example, 
if $\varchi = \dbx p$, then $\varchi(\{\phi,\psi\})= \{ \dbx \phi, \dbx \psi\}$. 

The complexity of formulas is defined as usual, where connectives and modal operators increase the complexity by 1. We define a partial order $\sml$ on sequents, based on the Dershowitz-Manna well-ordering 
$\sml_{dm}$ on multisets, in the usual way: $S_1 \sml S_2$ exactly if $S_1^a \cup S_1^s \sml_{dm} S_2^a \cup S_2^s$. Here $\sml_{dm}$ is the reflexive transitive closure of the ordering $\smll_{dm}^-$ between multisets, where $\Ga \smll_{dm}^-\Pi$ precisely if $\Ga$ is the result of replacing a formula in $\Pi$ by finitely many formulas of lower complexity than that formula. Furthermore, $S_1 \smll S_2$ precisely if 
$S_1 \sml S_2$ and $S_1$ and $S_2$ are not equal as sequents.

\subsection{Gentzen calculi}

A rule $R$ is an expression of the form 
\[
 \infer[R]{S_0}{S_1 & \dots & S_n}
\]
where the $S_i$ are sequents. $R_c$ denotes the formula $I(S_0)$ corresponding to the conclusion, and $R^a$ denotes the formula $\ben_i I(S_i)$ corresponding to the conjunction of the premisses. An {\em axiom} is a rule with no premisses, thus in our view, axioms are rules. 

Given an extension $\G$ of $\Gp$, to be defined below, and a rule $R$, we denote the calculus $\G+R$ by $\GR$. In the case of $\Gp$, we leave out the ``p'' and write $\GpR$ instead of {\sf G3pR}.  

Rule $R$ is a {\em box rule} if it satisfies:  
\begin{itemize}
\itm The conclusion of $R$ is of the form $\bx S \cdot (\Ga \seq \Sig)$ for some 
     sequent $S$ and two multisets $\Ga,\Sig$ not occurring, as multiset symbols, in $S$ nor in the premisses of $R$.  
\itm All premisses of $R$ consist of subformulas of formulas in $S$. 
\itm If an instance of $R$ is of the form 
\[
 \infer{(\Ga \seq \Sig) \cdot\bx (S_0 \cdot S_0' \cdot S_0')}{S_1 & \dots & S_n}
\]
there are sequents $S_i'$ and $S_i''$ such that $S_i = S_i'\cdot S_i'' \cdot S_i''$ and 
\[
 \infer{(\Ga \seq \Sig) \cdot\bx (S_0 \cdot S_0')}{S_1' \cdot S_1'' & \dots & S_n' \cdot S_n''}
\]
is an instance of $R$ as well. 
\end{itemize}

The last requirement guarantees that when box rules are added to a sequent calculus, closure under weakening and contraction is preserved, as will be proved in Lemma~\ref{lemweakcont}. 

Examples of a box rule (left) and a rule that is not a box rule (right):  
\[
 \infer{\Pi,\bx p \seq \De}{p \seq } \ \ \ \ \ \ \ 
 \infer{\Pi,\bx p \seq \De}{\Pi, p \seq \De}
\]
Important in this paper are the box rules $R_\Kf$, $R_\GL$ and $R_\Grz$, which are, respectively,  
\[
 \infer[R_{\Kf}]{\Pi,\bx \Ga \seq \bx \phi,\De}{\dbx \Ga \seq \phi} \ \ \ \ 
 \infer[R_{\GL}]{\Pi,\bx \Ga \seq \bx \phi,\De}{\dbx \Ga,\bx \phi \seq \phi} \ \ \ \ 
 \infer[R_{\Grz}]{\Pi,\bx\Ga \seq \bx\phi,\De}{\bx\Ga,\bx(\phi\imp\bx\phi) \seq \phi}
\]
A {\em Gentzen calculus} or a {\em sequent calculus} is a finite set of rules. In this paper we only consider Gentzen calculi of the form $\Gp+\calr$ for some set of box rules $\calr$, where $\Gp$ is given as follows. 

\subsection*{The Gentzen calculus $\Gp$} 
 \label{Gthp}
\renewcommand*{\arraystretch}{2.5}
\[
 \begin{array}{lll}
  \infer[{\it Ax}\ \text{($p$ an atom)}]{\Ga,p \seq p,\De}{} && \infer[L\bot]{\Ga,\bot\seq \De}{} \\
  \infer[L\en]{\Ga,\phi \en \psi \seq \De}{\Ga,\phi,\psi \seq \De} &&
   \infer[R\en]{\Ga \seq \phi \en \psi,\De}{\Ga \seq \phi,\De & \Ga \seq \psi, \De} \\ 
  \infer[L\of]{\Ga,\phi \of \psi \seq \De}{\Ga,\phi \seq \De & \Ga,\psi \seq \De} &&
   \infer[R\of]{\Ga \seq \phi\of \psi, \De}{\Ga \seq \phi, \psi, \De} \\
  \infer[L\!\imp]{\Ga, \phi \imp \psi \seq \De}{\Ga\seq \phi,\De & \Ga, \psi \seq \De} && 
   \infer[R\!\imp]{\Ga \seq \phi \imp \psi,\De}{\Ga, \phi \seq \psi,\De} 
 \end{array}
\]
\renewcommand*{\arraystretch}{1}

A {\em derivation tree} for $S$ in a calculus $\G$ is a finite tree labelled with sequents,  
where the root is labelled with $S$, and 
every inner node (not a leaf) with all its parent(s) forms an instance of a rule in $\G$. 
A {\em derivation} or {\em (standard) proof} of $S$ in $\G$ is a derivation tree for which all the 
leafs are axioms. We write $\af_\G S$ if sequent $S$ has a derivation in $\G$, and when $\G$ is clear from the context we write $\af$ instead of $\af_\G$. We write $\af_d S$ if $S$ has a proof of depth (length of the longest branch of the derivation tree) at most $d$. 

\begin{theorem}
 \citep{avron84}\label{thmcompleteness}
$\af_{\GGL} S$ if and only if $\af_{\GL} I(S)$. 
\end{theorem}

A substitution $\sig$ is a map from formulas in $\lang$ to formulas in $\lang$ that commutes with the connectives and the modal operator. $\sig S$ denotes the sequent $\{ \sig \phi \mid \phi \in S^a\} \seq \{ \sig \phi \mid \phi \in S^s \}$. We say that $\phi$ 
{\em admissibly derives} $\psi$ in a logic $\lgc$, notation $\phi \adm_\lgc \psi$, if for every substitution $\sig$, if $\af_\lgc \sig \phi$, then $\af_\lgc \sig \psi$. For sequents $S$ and $S'$, $S$ {\em admissibly derives} $S'$ in a calculus $\G$, notation $S \adm_\G S'$, if for every substitution $\sig$, if $\af_\G \sig S$, then $\af_\G \sig S'$. Lemmas~\ref{lemloeb} below provides a typical example of admissibility. Clearly, if $\phi \af_\lgc \psi$, then $\phi \adm_\lgc \psi$. But the converse is not always the case, more on this topic can be found in \cite{jerabek05}.  

A leaf with label $S$ for which there is a node at its branch properly below it with the 
same label $S$ is {\em circular}.
A {\em circular derivation} or {\em circular proof} of $S$ in a calculus $\G$ is a derivation tree for $S$ in which every leaf either is an axiom of $\G$ or is circular. 
We write $\afinf_\G S$ if sequent $S$ has a circular derivation in $\G$. 
A circular derivation is in particular a derivation tree. 

Clearly, $\af_\G S$ implies $\afinf_\G S$, but not vice versa, as the following circular proof of the sequent version of {\em L\"ob's principle} shows.
\[
 \infer[R_\Kf]{\bx(\bx\phi\imp\phi) \seq \bx\phi}{
  \infer[L\!\imp]{\dbx (\bx\phi\imp\phi) \seq \phi}{
   \bx(\bx\phi\imp\phi) \seq \bx\phi & \bx(\bx\phi\imp\phi),\phi \seq \phi } }
\]
Thus we can conclude that the sequent version of L\"ob's principle has a circular proof in $\Kf$: 
$\afinf_\Kf \bx(\bx\phi\imp\phi) \seq \bx\phi$. As the principle is not provable in $\Kf$ this shows that $\afinf_\Kf$ is strictly stronger than $\af_\Kf$. 
One of the corollaries of the main theorem of this note is that, actually, a sequent has a proof in $\GL$ if and only if it has a circular proof in $\Kf$.  

If we weaken the requirement of circular leafs to: there is a node at its branch properly below it with a label that has the same set-sequent as $S$, the system is no longer sound, as the following circular proof shows. 
\[
 \infer[L\en]{\phi \en \psi,\phi \en \psi,\phi,\psi \seq \,}{\phi \en \psi,\phi,\psi,\phi,\psi  \seq \,}
\]

\begin{lemma}
 \label{lemloeb}
$\bx\Pi,\dbx\Sig,\bx \phi \seq \phi \adm_{\GGL} \bx\Pi,\dbx\Sig \seq \phi$. 
\end{lemma}
\begin{proof}
The following steps prove the lemma, using in the third step that $S \adm_\GGL \bx S$ for any $S$. 
\[
 \begin{array}{ll}
  \bx\Pi,\dbx\Sig,\bx \phi \seq \phi & \ \af_{\GGL} \\
  \dbx\Pi,\dbx\Sig \seq \bx \phi \imp \phi & \adm_{\GGL} \\
  \bx\Pi,\bx\Sig \seq \bx(\bx\phi \imp \phi) & \adm_{\GGL} \\
  \bx\Pi,\bx\Sig \seq \bx\phi & \ \af_{\GGL} \text{(using first line)} \\ 
  \bx\Pi,\dbx\Sig \seq \phi.
 \end{array}
\]
\end{proof}

\subsection{Weakening and contraction}

\begin{lemma} (Inversion Lemma)
 \label{leminversion}
If $\calr$ is a set of box rules, then in $\Gp+\calr$ the following holds.

\begin{enumerate}
\item $\af_d \Ga, \phi \en \psi \seq \De$ implies $\af_d \Ga, \phi,\psi \seq \De$. 
\item $\af_d \Ga, \phi_0 \of \phi_1 \seq \De$ implies $\af_d \Ga, \phi_i \seq \De$ for $i=0,1$. 
\item $\af_d \Ga, \phi \imp \psi \seq \De$ implies $\af_d \Ga, \psi \seq \De$ and $\af_d\Ga\seq \phi,\De$.
\item $\af_d \Ga \seq \phi_0 \en \phi_1,\De$ implies $\af_d \Ga \seq \psi_i,\De$ for $i=0,1$. 
\item $\af_d \Ga \seq \phi \of \psi,\De$ implies $\af_d \Ga\seq \phi,\psi, \De$. 
\item $\af_d \Ga \seq \phi \imp \psi,\De$ implies $\af_d \Ga, \phi \seq \psi, \De$. 
\end{enumerate}

\end{lemma}
\begin{proof}
Analogues to the proof of Lemma 5.1.6 in \citep{troelstra&schwichtenberg96}. 
With induction to $d$. The case that $d=1$ is straightforward. In the induction step we consider the last inference of the derivation and distinguish by cases. For inferences that are instances of rules in $\Gp$ we reason as in \citep{troelstra&schwichtenberg96}. For an instance $S_1 \dots S_n/(\Pi \seq \Sig) \cdot \bx S_0$ of a box rule $R$ in $\calr$, it follows that any formula $\phi$ in the conclusion that is not boxed  
can be replaced by any formula $\psi$ and still have a valid proof, thus proving that also in this case all six properties in the lemma hold. 
\end{proof}

\begin{lemma} 
 \label{lemweakcont}
For any set of box rules $\calr$ weakening and contraction are depth preserving admissible in the calculus $\Gp+\calr$: In $\Gp+\calr$, for any sequents $S$ and $S'$ the following holds. 

\begin{itemize}
\itm If $\af_d S$, then $\af_d S'\cdot S$.  
\itm If $\af_d S'\cdot S'\cdot S$, then $\af_d S'\cdot S$. 
\end{itemize}

\end{lemma}
\begin{proof}
We prove the lemma with induction to $d$. The proof for weakening is straightforward and therefore left to the reader. The key ingredient is the observation that for any instance 
\[
 \infer{S_0}{S_1 & \dots & S_n}
\]
of a box rule $R$ and any sequent $S$, 
\[
 \infer{S \cdot S_0}{S_1 & \dots & S_n} 
\]
is an instance of $R$ as well. 

We turn to contraction. Suppose $\Gp+\calr$ derives $S' \cdot S' \cdot S$. 
Clearly, it suffices to treat the case that $S'$ consist of a single formula, say $\phi$. We treat the case that $S' = (\phi \seq \,)$, the other case being analogous. 
If $d=1$, then $S' \cdot S' \cdot S$ is an instance of an axiom. If it is an axiom of $\Gp$, inspection of the possible axioms shows that whence $S' \cdot S$ is an instance of that axiom too. If the axiom belongs to $\calr$, then the third requirement in the definition of box rules implies that $S' \cdot S$ is an instance of the axiom too. 

If $d>1$, consider the last inference 
\begin{equation}
 \label{eqcontraction}
 \infer{S' \cdot S' \cdot S}{S_1 & \dots & S_n}
\end{equation}
of the derivation. If it is an instance of a rule $R$ in $\calr$, $S' \cdot S' \cdot S = (\Ga \seq \Sig)\cdot \bx S_0$ for some $\Ga,\Sig$ and $S_0$. There are several cases to consider: (1) $\phi$ occurs twice in $\bx S_0^a$ or (2) $\phi$ occurs twice in $\Ga$ or (3) 
$\phi$ occurs in $S_0^a$ and $\Ga$. 

In case (1) the third requirement in the definition of box rules implies that there exist $S_i'$ and $S_i''$ such that $S_i = S_i'\cdot S_i'\cdot S_i''$ and 
\[
 \infer{S' \cdot S}{S_1'\cdot S_1'' & \dots & S_n' \cdot S_n''}
\]
is an instance of $R$. By the induction hypothesis, the $S_i'\cdot S_i''$ have proofs of depth smaller than $d$, which proofs that $S' \cdot S$ has proof of depth at most $d$.  
In cases (2) and (3) it follows that 
\[
  \infer{S' \cdot S}{S_1 & \dots & S_n}
\]
is an instance of $R$, and we are done immediately. 

If \eqref{eqcontraction} is an instance of a rule of $\Gp$, we have to distinguish by cases. We treat the left implication rule. Therefore assume \eqref{eqcontraction} is of the form 
\[
 \infer[R]{\Ga,\phi \imp \psi \seq \De}{\Ga \seq \phi,\De & \Ga,\psi \seq \De}
\]
where either $\De$ or $\Ga$ contains a formula twice or $\Ga$ contains $\phi\imp \psi$. In the first two cases the induction hypothesis immediately applies. In the last case, by applying Lemma~\ref{leminversion} to the two premisses, it follows that  
$\Ga\backslash\{\phi \imp \psi\} \seq \phi,\phi,\De$ and 
$\Ga\backslash\{\phi \imp \psi\},\psi,\psi \seq \De$ have proofs of depth $<d$. Hence so do 
$\Ga\backslash\{\phi \imp \psi\} \seq \phi,\De$ and 
$\Ga\backslash\{\phi \imp \psi\},\psi \seq \De$ by the induction hypothesis. An application of $L\imp$ gives $\Ga\backslash\{\phi \imp \psi\},\phi\imp\psi \seq \De$. 

\end{proof}

\begin{corollary}
 \label{corstructural}
Weakening and contraction are admissible in $\GKf$ and $\GGL$. 
\end{corollary}

In this paper we do not need the admissibility of cut in $\GKf$ and $\GGL$, but it is worth mentioning that the rule is indeed admissible.  For a proof, see, for example, \citep{avron84}.

\subsection{Ordered rules and proofs}

A rule is {\em ordered} if all its premisses are $\smll$--lower than its conclusion and consist solely of subformulas of formulas in the conclusion. 

An instance of a rule is a {\em set}--instance if the premisses are set--sequents. Given a set of rules $\calr$, a proof is {\em $\calr$--set} if every instance in the proof of a rule in $\calr$ is a set--instance. 

Given a calculus \G, denote by $\calr_\G$ the set of those rules in $\G$ in which the premisses are not $\smll$-lower than the conclusion. 
A calculus $\G$ is {\em ordered} if  every provable sequent has a proof that is $\calr_\G$-set.  

A rule $R$ is {\em slim} if for every instance $S_1 \dots S_n/S$ of it,  
$S_1^* \dots S_n^*/S$ is an instance of $R$ as well. 
Observe that both $R_\GL$ and $R_\Kf$ are slim rules. 


\begin{lemma} 
 \label{lemfinite}
For every set $\calr$ of box rules that are slim or ordered: for any sequent $S$ provable in $\Gp+\calr$, there is a finite set of sequents $\cals$ such that in any proof of $S$ in $\Gp+\calr$ that is $\calr$--set, only sequents in $\cals$ occur. 
\end{lemma}
\begin{proof}
Let $\cals'$ consist of all set--sequents that consist of subformulas of formulas in $S$. $\cals$ denote the $\cals'$ union all sequents that are $\smll$-lower than a sequent in $\cals'$. Because of the subformula property that box rules as well as rules in $\Gp$ satisfy, every $\calr$--set proof of $S$ contains only sequents in $\cals$. 
\end{proof}

\begin{lemma} 
 \label{lemsetproofs}
For every extension $\G$ of $\Gp$ by ordered box rules and for every set $\calr$ of slim box rules: for every proof in $\G +\calr$, there exists an $\calr$--set proof in $\G +\calr$ of the same endsequent of depth no greater than the original proof.   
\end{lemma}
\begin{proof}
Consider a proof in $\G' = \G +\calr$. 
With induction on the depth $d(\cald)$ of the lowest inferences that violate that $\cald$ is  $\calr$--set, with a subinduction to the number $m(\cald)$ of those lowest inferences that violate that $\cald$ is $\calr$--set. If $d(\cald)=0$, then $\cald$ is $\calr$--set and there is nothing to prove. 
 
If $d(\cald)>0$, consider an inference 
\[
 \infer{S}{S_1 & \dots & S_n} 
\]
at depth $d(\cald)$ which is an application of a rule $R \in \calr$ such that not all $S_i$ are set--sequents. As $R$ is closed under contraction, Lemma~\ref{lemweakcont} implies that the sequent $S_i^*$ has a proof of the same or lower depth than the proof of $S_i$. Since $R$ is a set--rule, this implies $S$ has a proof of depth $\leq d(\cald)$ in which the last inference is a set--instance of $R$. Replacing the subproof of $S$ in $\cald$ by this proof results in a proof $\cald'$ with the same endsequent as $\cald$ for which either $d(\cald')<d(\cald)$, or $d(\cald')=d(\cald)$ and $m(\cald') < m(\cald)$. In both cases the induction hypothesis applies and we obtain a proof of the endsequent of $\cald$ that is $\calr$--set. 
\end{proof}

\section{From standard proofs to circular proofs}
 
\begin{lemma}
 \label{lemmain}
For every extension $\G$ of $\Gp$ by ordered box rules: if $R_1,R_2$ are slim box rules such that $R_1^a \adm_{\GR_1} R_2^a$ and $R_1^c=R_2^c$, then $\af_{\GR_1} S$ implies $\afinf_{\GR_2} S$. 
\end{lemma}
\begin{proof}
First we need to introduce some terminology. 
Given a derivation $\cald$, let $h^R_\cald$ denote the height of the lowest application of $R$ 
in $\cald$, where the height on a application of a rule $R$ is the number of nodes from the root of the tree to the conclusion of that application. If $\cald$ does not contain applications of $R$ we put $h^R_\cald=0$. With $n^R_\cald$ we denote the number of applications of $R$ at height $h^R_\cald$ in $\cald$. 

Let $\G$ be $\Gp$ extended by $R_1$ and $R_2$. 
Suppose $\af_{\GR_1} S$ and let $\cald_0$ be an $\{R_1,R_2\}$--set proof of $S$ in $\GR_1$, which exists by the previous lemma. We construct a sequence 
$\cald_0,\cald_1,\cald_2, \dots$ of $\{R_1,R_2\}$--set proofs in $\G$ with the following properties, where 
$h_i = h_{\cald_{i}}^{R_1}$ and $n_i = n_{\cald_i}^{R_1}$. 
For every $i$ either $n_i =0$, or 
$h_{i+1} = h_i$ and $n_{i+1} < n_i$, or 
$h_{i+1} > h_i$. In no $\cald_i$ there are applications of 
$R_2$ above applications of $R_1$. In other words, subproofs that end in an application of 
$R_1$, are proofs in $\GR_1$. 

If $\cald_i$ contains no application of $R_1$, then the sequence stops at $\cald_i$ with $n_i =0$. Otherwise 
consider the leftmost application of $R_1$ at height $h_i$ and let $S_0$ and $S_1$ be its conclusion and its premiss, respectively. The subproof of $S_0$ therefore is a proof in $\GR_1$. Since $S_1=R_1^a \adm_{\GR_1} R_2^a$, there exists a proof in $\GR_1$ of $R_2^a$. As $R_1$ is a slim rule, Lemmas~\ref{lemweakcont} and \ref{lemsetproofs} imply that there is an $\{R_1\}$--set proof of $(R_2^a)^*$ in $\GR_1$. Let $\cald$ denote this proof followed  by an application of $R_2$. Thus $\cald$ is an $\{R_1,R_2\}$--set proof of $S_0$. Let $\cald_{i+1}$ be the result of replacing the considered subproof of $S_0$ by $\cald$. We show that it has the required properties. 

That there is no application of $R_2$ above applications of $R_1$ is clear. If $n_i >1$, 
then $n_{i+1} = n_i - 1 < n_i$ and $h_i=h_{i+1}$. If, on the other hand, $n_i=1$, then $h_{i+1} > h_i$ or $n_{i+1}=0$. 
This proves that a sequence of proofs as described above can be constructed. 

Since all $\cald_i$ are $\{R_1,R_2\}$--set proofs in $\G$ it follows from Lemma~\ref{lemfinite} that there exists a finite set of sequents $\cals$ such that every sequent that occurs in some $\cald_i$ belongs to $\cals$. There are two possibilities: the sequence of the $\cald_i$ is finite or it is infinite. It follows from the construction 
that in the first case the last proof in the sequence does not contain applications of $R_1$. Thus it is a proof in $\GR_2$. Hence $\af_{\GR_2} S$ and therefore $\afinf_{\GR_2} S$. If the sequence is infinite,
Consider $\cald_i$ for an $i$ for which $h_i$ is larger than the number of sequents in $\cals$. The length of any branch in $\cald_i$ is either greater than $h_i$ or at most $h_i$. In the last case, it cannot contain applications of $R_1$. In the first case, the sequent at height $h_{i+1}$ has to occur at that branch at a height lower than $h_i$ as well. Therefore, if we cut away all nodes at height $h_{i+2}$ and  higher we obtain a circular proof of $S$ in $\GR_2$. 
\end{proof}

\section{From circular proofs to standard proofs}

\begin{lemma}
 \label{leminbetween}
For every extension $\G$ of $\Gp$ by ordered box rules, 
if in a proof of a sequent in $\G +R$ there is a branch with two nodes with the same label, then there is an application of $R$ between these two occurrences along the branch. 
\end{lemma}
\begin{proof}
In all rules in $\G$ the premisses are $\smll$--lower than the conclusion. 
\end{proof}
 





Given a calculus $\G$ and two rules $R_1$ and $R_2$, calculus $\GR_2$ is the {\em circular companion} of calculus $\GR_1$ if 
there exist formulas $\varchi(p)$ and $\eta(p)$ such that for any instance $S_1 \dots S_n/S_0$ of $R_2$ and  for all multisets $\Pi$ and $\Sig$ (recall that $\varchi(S)$ stands for $\varchi(I(S))$ and $\varchi(\Ga)$ for $\{ \varchi(\phi) \mid \phi\in\Ga \}$, and likewise for $\eta$): 


\begin{itemize}
\itm $R_1^a \adm_{\GR_1} R_2^a$ and $R_1^c=R_2^c$; 
\itm $\af_{\GR_1}\, \eta(\phi) \imp \phi$ for all formulas $\phi$; 
\itm $\{ \varchi(\Pi),\eta(\Sig),S_i^a \seq S_i^s \mid 1\leq i\leq n\} \adm_{\GR_1}\, 
     \varchi(\Pi \cup \Sig),S_0^a \seq S_0^s$;
\itm $\varchi(\Pi),\eta(\Sig),\varchi(S),S^a \seq S^s \adm_{\GR_1}\, 
     \varchi(\Pi), \eta(\Sig),S^a \seq S^s$ for any sequent $S$;
\itm for every instance $S_1' \dots S_n'/S_0'$ of a rule in $\G$, 
     $S'\cdot S_1' \dots S'\cdot S_n'/S' \cdot S_0'$ is an instance as well, for 
     $S'$ of the form $(\chi(\Pi),\eta(\Sig) \seq \,)$. 
\end{itemize}

\begin{remark}
 \label{remcompanion}
The last two requirements in the definition of circular companions imply that for such companions also holds:
\[
 \varchi(\Pi),\eta(\Sig),\eta(S_0),S_i^a \seq S_i^s \mid 1\leq i\leq n\} \adm_{\GR_1}\, 
     \varchi(\Pi \cup \Sig),S_0^a \seq S_0^s. 
\]
\end{remark}

\begin{remark}
 \label{remcircular}
$\GKf$ is the circular companion of $\GGL$ by taking $\eta(p)=\dbx p$ and $\varchi(p)=\bx p$. In fact, for any extension $\G$ of $\Gp$, $\G+R_{\Kf}$ is the circular companion of $\G+R_{\GL}$ for the same $\eta$ and $\varchi$. That the second requirement holds is trivial. For the third one the following observations suffice, recalling that $\bx(\bx\Ga \seq \bx\phi)$ denotes $\bx(\ben\bx\Ga \imp \bx\phi)$.  
\[
 \begin{array}{ll}
  \bx\Pi,\dbx\Sig,\dbx\Ga \seq \phi & \adm_{\GGL} \\
  \dbx\Pi,\dbx\Sig,\dbx\Ga,\bx\phi \seq \phi & \adm_{\GGL} \\
  \bx\Pi,\bx\Sig,\bx\Ga \seq \bx\phi.
 \end{array}
\]
The fourth requirement follows from Lemma~\ref{lemloeb} with $\phi = I(S)$, and the first requirement is left to the reader. 
\end{remark}

\begin{lemma}
 \label{lemmaintoo}
For every extension $\G$ of $\Gp$ by ordered box rules: 
if $R_1,R_2$ are slim box rules such that $\GR_2$ is the  circular companion of $\GR_1$, 
then $\afinf_{\GR_2} S$ implies $\af_{\GR_1} S$.
\end{lemma}
\begin{proof} 
Given a derivation tree $\cald$ in $\GR_2$, a leaf labelled with sequent $S$ is an {\em assumption leaf\/} if it is not circular and $S$ is not an axiom.   
Denote by ${\it ap}_\cald$ and ${\it nap}_\cald$ the sets of formulas of the form $I(S)$, where $S$ is the label of an assumption leaf that has, respectively does not have, an application of $R_2$ along its branch. 

Suppose $\GR_2$ is the  circular companion of $\GR_1$ and let $(\varchi,\eta)$ be the witness of it.
We prove with induction to the height of a circular derivation tree $\cald$ in $\GR_2$ with root $S$:
\begin{equation}
 \label{eqg2tog1}
 \af_{\GR_1} \varchi({\it ap}_\cald), \eta({\it nap}_\cald),S^a \seq S^s.
\end{equation}
Since for a circular proof in $\GR_2$, both ${\it ap}_\cald$ and ${\it nap}_\cald$ are empty, this will prove the lemma. 

If $\cald$ consists of one sequent only, it is either an axiom of $\GR_2$, in which case \eqref{eqg2tog1} clearly holds, or it is an assumption leaf with no application of $R_2$ along its branch, which also implies \eqref{eqg2tog1} because $\eta(\phi)$ implies $\phi$ for all formulas $\phi$. 

Suppose the height of $\cald$ is greater than one and suppose the last inference of $\cald$ is 
an application of a rule $R$ and let $S_1,\dots,S_m$ be its premisses. The induction hypothesis and the fact that $\GR_1$ is closed under weakening gives for every $i$:
\begin{equation}
 \label{eqindhyp}
 \af_{\GR_1} \varchi(\bigcup_{i=1}^m{\it ap}_{\cald_i}), 
 \eta(\bigcup_{i=1}^m{\it nap}_{\cald_i}), S_i^a \seq S_i^s, 
\end{equation}
We distinguish the cases that any leaf in $\cald$ that is circular is circular in one of the $\cald_i$, and the opposite case. 
In the first case, if $R$ is one of the rules of $\G$, then ${\it ap}_\cald$ is equal to $\bigcup\{ {\it ap}_{\cald_i} \mid i\leq m\}$, and similarly for ${\it nap}_\cald$. Therefore \eqref{eqg2tog1} follows from the last requirement of circular companions and an application of $R$ to \eqref{eqindhyp}.  
If $R=R_2$, then ${\it nap}_{\cald}$ is empty and 
\[
 {\it ap}_{\cald} = \bigcup_{i=1}^m ({\it ap}_{\cald_i} \cup {\it nap}_{\cald_i}). 
\]
As $\GR_2$ is the  circular companion of $\GR_1$, it follows that 
$\af_{\GR_1} \varchi({\it ap}_{\cald}), S^a \seq S^s$, which implies \eqref{eqg2tog1}. 

Next, consider the case that in $\cald$ there is a circular leaf that is not circular in any of the 
$\cald_i$. Note that all such leafs are labelled with the same sequent as the endsequent of the proof, $S$, and that they may become assumption leafs in the $\cald_i$. Hence 
\[
 {\it ap}_{\cald} \cup {\it nap}_{\cald} = 
 (\bigcup_{i=1}^m ({\it ap}_{\cald_i} \cup {\it nap}_{\cald_i})) \backslash \{S\}.
\]
First consider the case that $R$ is one of the rules of $\G$. 
By Lemma~\ref{leminbetween} it follows that there is an application of $R_2$ along branches that have leaf $S$, which means that if $S$ occurs in $\bigcup_{i=1}^m ({\it ap}_{\cald_i} \cup {\it nap}_{\cald_i})$, it occurs in $\bigcup_{i=1}^m {\it ap}_{\cald_i}$, and therefore as $\varchi(S)$ in \eqref{eqindhyp}. 
An application of $R$ to \eqref{eqindhyp} gives 
\[
 \af_{\GR_1} \varchi({\it ap}_{\cald}), \eta({\it nap}_{\cald}), \varchi(S), S^a \seq S^s.
\] 
The fact that $\GR_2$ is the  circular companion of $\GR_1$ now implies \eqref{eqg2tog1}. 

If, on the other hand, $R=R_2$, then ${\it nap}_{\cald}$ is empty and for any $i$, $S$ may be in either ${\it ap}_{\cald_i}$ or ${\it nap}_{\cald_i}$. And thus appear as $\varchi(S)$ or $\eta(S)$ in \eqref{eqindhyp}. Then the fact that $\GR_2$ is the  circular companion of $\GR_1$ and Remark~\ref{remcompanion} imply \eqref{eqg2tog1}. 
\end{proof}

\section{Standard proofs versus circular proofs}
 \label{secresults}
The results in the previous section lead to a sufficient condition for being the circular companion of a logic, from which Shamkanov's results follow. 
 
\begin{theorem}
 \label{thmmain}
For every extension $\G$ of $\Gp$ by ordered box rules and for any slim box rules $R_1$ and $R_2$ such that $\GR_2$ is the circular companion of $\GR_1$:
\[
 \af_{\GR_1} S \text{ if and only if } \afinf_{\GR_2} S.
\]
\end{theorem}

\begin{theorem}
For every extension $\G$ of $\Gp$ by ordered box rules: 
\[
 \af_{\G +R_{\GL}} S \text{ if and only if } \afinf_{\G +R_{\Kf}} S.
\]
\end{theorem}
\begin{proof}
Use Theorem~\ref{thmmain} and Remark~\ref{remcircular} with $R_1=R_{\GL}$ and $R_2=R_{\Kf}$. 
\end{proof}

These theorems immediately give us Shamkanov's Theorem: 

\begin{corollary}
$\af_{\GGL} S$ if and only if $\afinf_{\GKf} S$.
\end{corollary}

\subsection{Intuitionistic modal logics}
Inspection of the proofs of the theorems above show that they also hold when $\Gp$ is replaced by 
one of the standard single-conclusion Gentzen calculi for intuitionistic logic without structural rules, such as the propositional part of {\sf G3i} from \citep{troelstra&schwichtenberg96}, or Dyckhoff's calculus \citep{dyckhoff92}. If $\iGGL$ and $\iGKf$ denotes the extension of one of Dyckhoff's calculus by the single conclusion versions of the rules $R_{\GL}$ and $R_{\GL}$, respectively, we can conclude the following. 

\begin{theorem}
 \label{thmmainint}
$\af_{\iGGL} S$ if and only if $\afinf_{\iGKf} S$.
\end{theorem}

\subsection{Grzegorczyk logic}
Recall that there is a cut--free sequent calculus for $\Sf$, which consists of $\Gp$ plus $R_{\Kf}$ and $R_{\sf T}$, where $R_{\sf T}$ is the rule 
\[
 \infer[R_{\sf T}.]{\Ga,\bx\phi \seq \De}{\Ga, \phi \seq \De}
\]
In \citep{avron84} it is shown that the calculus $\Gp+R_{\Grz}+R_{\sf T}$ has cut-elimination. In fact, it is shown that a variant of $\G+R_{\Grz}+R_{\sf T}$ with explicit weakening has cut-elimination, but it is not hard to see that this implies the former result. 

Note that $R_{\sf T}$ is an ordered rule. It is not a box rule, but it is not hard to see that the reasoning in the previous proofs about box rules applies to this rule as well. We therefore have the following. 

\begin{corollary}
$\af_{\Gp+R_{\Grz}+R_{\sf T}} S$ implies $\afinf_{\GSf} S$.
\end{corollary}
\begin{proof}
Proved in a similar way as Lemma~\ref{lemmain} with $\G=\Gp+R_{\sf T}$, $R_1=R_{\Grz}$ and $R_2=R_{\Kf}$. 
\end{proof}

The converse, however, does not hold, since L\"ob's principle has a circular proof in $\Kf$, as we saw, but is not provable in Grzegorczyk logic.







\end{document}